\documentclass[11pt]{amsart}

\usepackage{amssymb,enumitem}

\usepackage{amsthm}

\usepackage[colorlinks=false]{hyperref}

\theoremstyle{plain}
\newtheorem{theorem}{Theorem}[section]
\newtheorem{lemma}[theorem]{Lemma}
\newtheorem{corollary}[theorem]{Corollary}
\newtheorem{proposition}[theorem]{Proposition}
\newtheorem{conjecture}[theorem]{Conjecture}

\theoremstyle{definition}
\newtheorem{claim}{}[theorem]
\newtheorem{remark}{Remark}

\newcommand{\mcal}[1]{\ensuremath{\mathcal{#1}}}
\newcommand{\mbf}[1]{\ensuremath{\mathbf{#1}}}
\newcommand{\mbb}[1]{\ensuremath{\mathbb{#1}}}
\newcommand{\gf}[1]{\ensuremath{\mathrm{GF}(#1)}}
\newcommand{\pg}[1]{\ensuremath{\mathrm{PG}(#1)}}
\newcommand{\fr}[1]{\ensuremath{\mathrm{Fr}(#1)}}
\newcommand{\var}[1]{\ensuremath{\mathrm{Var}(#1)}}
\newcommand{\kin}[1]{\ensuremath{\mathrm{Kin}(#1)}}

\newcommand{\dash}{\nobreakdash-\hspace{0pt}}
\newcommand{\ba}{\backslash}
\newcommand{\cl}{\ensuremath{\operatorname{cl}}}

\newcommand{\PH}{$M$\protect\dash logic}

\title
[The missing axiom of matroid theory]
{Is the missing axiom of matroid theory lost forever?}

\author[Mayhew]{Dillon Mayhew}
\address{School of Mathematics, Statistics, and
Operations Research,
Victoria University,
Wellington,
New Zealand}
\email{dillon.mayhew@msor.vuw.ac.nz}

\author[Newman]{Mike Newman}
\address{Department of Mathematics and Statistics\\
University of Ottawa\\
Ottawa\\
Canada}
\email{mnewman@uottawa.ca}

\author[Whittle]{Geoff Whittle}
\address{School of Mathematics, Statistics, and
Operations Research,
Victoria University,
Wellington,
New Zealand}
\email{geoff.whittle@msor.vuw.ac.nz}


\date{\today}

\begin{document}

\maketitle

\begin{abstract}
We conjecture that it is not possible to finitely
axiomatize matroid representability in
monadic second-order logic for matroids, and we describe
some partial progress towards this conjecture.
We present a collection of sentences
in monadic second-order logic and show that
it is possible to finitely axiomatize matroids using
only sentences in this collection.
Moreover, we can also axiomatize representability
over any fixed finite field (assuming Rota's conjecture holds).
We prove that it is not possible to finitely
axiomatize representability, or representability
over any fixed infinite field, using sentences from the
collection.
\end{abstract}

\section{Introduction}

The problem of characterizing representable
matroids is an old one.
(When we say that a matroid is representable, we mean that
it is representable over at least one field.)
Whitney discusses the task of
`characterizing systems which represent matrices'
in his foundational paper \cite{Whi35}.
From the context, it seems likely that
he means characterizing via a list of axioms.
We believe that this task will never be completed.
In other words, we conjecture that
`the missing axiom of matroid theory is lost forever'.

\begin{conjecture}
\label{conj1}
It is not possible to finitely
axiomatize representability for (finite)
matroids, using the same logical language as
the matroid axioms.
\end{conjecture}

Of course, this conjecture is not well-posed, unless
we specify exactly what the language of matroid
axioms is.
Certainly, a logic powerful enough to express the existence
of a matrix over a field whose columns have the required pattern
of independence would suffice to axiomatize representability,
but this logic would need to be much
more powerful than the language typically used to axiomatize
matroids.
Conjecture~\ref{conj} is an attempt to make
Conjecture~\ref{conj1} more precise.
In our main result (Theorem~\ref{main}), we demonstrate that
a weakened version of Conjecture~\ref{conj} is true.

In Section~\ref{logic} we develop
\emph{monadic second-order logic for matroids} (MSOL).
In MSOL we are allowed to quantify over variables that
are intended to represent elements or subsets of a ground set.
We admit the function that takes a subset
to its cardinality.
We allow ourselves the relations of equality,
element containment, set inclusion,
and the less-than-or-equal order on integers.
In addition, we also include a function, $r$, that takes subsets
of the ground set to non-negative integers.
This is intended to be interpreted as a rank
function.
As an example of the expressive capabilities of
MSOL, a matroid is paving if and only if its
rank function obeys the following sentence.
\[
\forall X_{1}\ |X_{1}| < r(E) \to r(X_{1})=|X_{1}|
\]
Hlin{\v{e}}n{\'y}~\cite{Hli06c} introduced a
logical language which was also called
monadic second-order logic for matroids.
It is easy to see that any sentence
in Hlin{\v{e}}n{\'y}'s language can be translated into
a sentence in our language.
On the other hand, our language includes the cardinality
function, while Hlin{\v{e}}n{\'y}'s does not.

The matroid rank axioms can be stated as sentences in
MSOL.
(Throughout the article we consider a matroid to
be a finite set equipped with a rank function.)
Moreover, for any matroid $N$,
we can construct a sentence in MSOL that will be true for a
matroid $M$ if and only if $M$ has an $N$\dash minor
(Proposition~\ref{1}).
This means that if Rota's conjecture is true, then
\gf{q}\dash representability can be finitely axiomatized
in MSOL, for any prime power $q$ (Lemma~\ref{Rota}).
We conjecture that it is impossible to finitely
axiomatize representability in MSOL.

\begin{conjecture}
\label{conj}
There is no finite set of sentences, \mcal{K},
in MSOL with the following property:
a finite set, $E^{\mcal{M}}$, equipped with a function
$r^{\mcal{M}}\colon \mcal{P}(E^{\mcal{M}})\to\mbb{Z}^{+}\cup\{0\}$,
is a representable matroid
if and only if $(E^{\mcal{M}},r^{\mcal{M}})$
satisfies the rank axioms and every sentence
in \mcal{K}.
\end{conjecture}

Our main result (Theorem~\ref{main}) shows
that Conjecture~\ref{conj} is true
if we insist that the sentences in \mcal{K} must
come from a restricted subset of MSOL.
We use the terminology \emph{\PH} to describe a
set of formulas in MSOL with constrained
quantification.
A formula in \PH\ must have the following
property: all variables representing subsets receive the
same type of quantifier (universal or existential),
and the same constraint applies to variables
representing elements.
We define \PH\ formally in Section~\ref{Mlogic}.

If \mcal{F} is a collection of fields, let $M(\mcal{F})$ be the
set of matroids that are representable over at least one field in
\mcal{F}.
Note that if \mcal{F} is the set of all fields, then
$M(\mcal{F})$ is the set of representable matroids.

\begin{theorem}
\label{main}
Let \mcal{F} be a set of fields that contains at least
one infinite field.
There does not exist a finite set, \mcal{K}, of sentences in \PH\
with the following property:
a finite set, $E^{\mcal{M}}$, equipped with a function
$r^{\mcal{M}}\colon \mcal{P}(E^{\mcal{M}})\to\mbb{Z}^{+}\cup\{0\}$,
is a matroid in $M(\mcal{F})$
if and only if $(E^{\mcal{M}},r^{\mcal{M}})$
satisfies the rank axioms and every sentence
in \mcal{K}.
\end{theorem}

Because any minor-closed class of matroids with a
finite number of excluded minors can be characterized
in $M$\dash logic (Corollary~\ref{cor}),
Theorem~\ref{main} gives us an alternative
proof and a strengthening of
\cite[Theorem~6.5.17]{Oxl11}.

\begin{corollary}
\label{cor2}
Let \mcal{F} be a set of fields that contains at least
one infinite field.
There are infinitely many excluded minors for
the class $M(\mcal{F})$.
\end{corollary}

Theorem~6.5.17 in \cite{Oxl11} concerns the case of
Corollary~\ref{cor2} when $|\mcal{F}|=1$.

We are interested in \PH\ because it provides a
separation between representability over finite fields
and infinite fields.
The axioms for matroid rank functions,
independent sets, bases, and spanning sets can
all be expressed using sentences in \PH\ (Section~\ref{axioms}).
Moreover, if Rota's conjecture holds, then
representability over a finite field
can be finitely axiomatized using sentences
in \PH\ (Lemma~\ref{Rota}).
Theorem~\ref{main} shows this is not the case for any infinite field.

The reader may be puzzled by our titular question,
since it is seemingly answered by
a well-known article due to V\'{a}mos \cite{Vam78}.
His article has the dramatic
title `The missing axiom of matroid theory is lost forever'.
When we examined the article, we were surprised to discover that
the words `matroid' and `axiom' in his title were
not used in the way we expected.
V\'{a}mos's result has been interpreted as
making a statement about finite matroids \cite{Gee08};
this is certainly what we anticipated.
But in the title of his paper, the word `matroid'
refers to an object that may be infinite.
Furthermore, it seems natural to use `axiom' to mean a
sentence constructed in the same language as the other matroid
axioms, but V\'{a}mos uses it to mean a sentence in a
first-order language which we call $V$\dash logic.
This logic is not
capable of expressing the matroid axioms
(as they are presented in \cite{Oxl11,Wel76,Whi35}).

We describe the matroid-like objects that V\'{a}mos
considers.
A \emph{pre-independence space} (see \cite{Oxl92b}) is a set
$E$, along with a family, \mcal{I}, of subsets of $E$,
called \emph{independent sets},
satisfying:
({\bf I1}) $\emptyset\in \mcal{I}$,
({\bf I2}) if $I\in\mcal{I}$ and $I'\subseteq I$, then
$I'\in \mcal{I}$, and
({\bf I3}) if $I$ and $I'$ are finite members of
\mcal{I} and $|I|=|I'|+1$, then there is an element
$x\in I-I'$ such that $I'\cup x\in\mcal{I}$.
Note that every finite pre-independence space is a (finite) matroid.
An \emph{independence space} is a pre-independence space
that satisfies
({\bf I4}) if $X\subseteq E$, and every finite subset
of $X$ is in \mcal{I}, then $X$ is in \mcal{I}.
The objects that V\'{a}mos calls matroids, and which we
call \emph{$V$\dash matroids}, are precisely the
pre-independence spaces with no infinite independent sets.

Let $(E,\mcal{I})$ be a $V$\dash matroid, and let
\mcal{S} be the family of infinite subsets of $E$
such that $X\in \mcal{S}$ if and only if every finite subset
of $X$ is independent.
Adding any arbitrary subset of \mcal{S} to \mcal{I}
produces a pre-independence space, and any pre-independence
space on the set $E$ containing \mcal{I} as its family of
finite independent sets can be produced in this way.
Thus $V$\dash matroids and independence spaces can be seen as
occupying opposite ends of the pre-independence space
spectrum:
a $V$\dash matroid is produced by adding the empty subset
of \mcal{S} to \mcal{I}, and an independence space is
produced by adding all of \mcal{S} to \mcal{I}.
As $V$\dash matroids are examples of
pre-independence spaces, they share some of the
peculiarities of this class.
For example, consider an infinite set $E$, and let
\mcal{I} be the collection of all finite subsets
of $E$ (c.f.\ \cite[Example~3.1.1]{Oxl92b}).
This is a $V$\dash matroid that has no
maximal independent sets, and no minimal
dependent sets.
Examples such as these perhaps explain why
pre-independence spaces have not been studied nearly
as often as independence spaces or $B$\dash matroids
(see \cite{Oxl92b}); and $V$\dash matroids have
been examined even less frequently.
So far as we are aware, V\'{a}mos's paper is the
only work in the literature that considers
$V$\dash matroids.

The first-order language that we call \emph{$V$\dash logic}
features, for every positive integer $n$,
an $n$\dash ary predicate, $I_{n}$.
The statement $I_{n}(x_{1},\ldots, x_{n})$ is designed to
be interpreted as saying that $\{x_{1},\ldots, x_{n}\}$
is in \mcal{I}.
Then $V$\dash matroids are models of a theory in $V$\dash logic.
Let \mcal{A} be a set of sentences in $V$\dash logic that
has the set of $V$\dash matroids as its models.
For example, \mcal{A} might contain, for every $n$,
the sentence
\[
\forall x_{1}\cdots \forall x_{n}\ 
I_{n}(x_{1},\ldots, x_{n})\to
\bigwedge_{\sigma\in S_{n}}
I_{n}(x_{\sigma(1)},\ldots, x_{\sigma(n)})
\]
to ensure that \mcal{I} consists of unordered sets.
It could also contain, for every $n$,
the sentence
\begin{multline*}
\forall x_{1}\cdots \forall x_{n}\forall y_{1}\cdots\forall y_{n+1}\
I_{n}(x_{1},\ldots, x_{n})\land
I_{n+1}(y_{1},\ldots, y_{n+1})\to\\
I_{n+1}(x_{1},\ldots,x_{n},y_{1})\lor\cdots\lor
I_{n+1}(x_{1},\ldots,x_{n},y_{n+1})
\end{multline*}
to ensure that ({\bf I3}) holds.

V\'{a}mos declares a $V$\dash matroid, $(E,\mcal{I})$, to be
\emph{representable} if there is a function
from $E$ to a vector space that preserves the
rank of finite subsets.
His theorem is as follows.

\begin{theorem}[\cite{Vam78}]
\label{vamos}
There is no sentence, $S$, in $V$\dash logic, such that
a $V$\dash matroid is representable
if and only if it satisfies $S$.
\end{theorem}

This theorem has no implications for
Conjecture~\ref{conj}, since it concerns a
different class of objects, and a different
logical language.
We feel that Conjecture~\ref{conj} quite
closely captures Whitney's question concerning
the characterization of representable matroids.
Therefore, our opinion is that Theorem~\ref{vamos}
does not answer the question posed by our title.

Obviously \mcal{A} is an infinite set of sentences, and
it is an easy exercise
to show that no finite set of sentences in $V$\dash logic has
the class of $V$\dash matroids as its set of models.
Given this fact, we are not surprised to
learn from Theorem~\ref{vamos}
that representable $V$\dash matroids cannot be characterized
with a single additional sentence.
In fact, we would go further, and conjecture that
no `natural' class of $V$\dash matroids can be
characterized by adding a single sentence to \mcal{A}.
(We are being deliberately vague about the meaning of
the word `natural'.)

In first-order logic, it is impossible to distinguish
between finite and infinite sets.
This is (presumably) the reason Theorem~\ref{vamos} is stated
in terms of $V$\dash matroids and not independence spaces,
as the axioms for independence spaces require us
to differentiate between finite and infinite sets, and
therefore cannot be expressed in $V$\dash logic.
V\'{a}mos's proof strategy relies upon the Compactness
Theorem of first-order logic.
One of the consequences of this is that we cannot
hope to rework the strategy to prove Conjecture~\ref{conj},
as that would require having an axiom in first-order
logic that restricts us to finite sets.

We conclude this introduction by briefly
describing the strategy for proving Theorem~\ref{main}.
The first step involves
developing an infinite family of matroids, each of
which is representable
over all infinite fields (Section~\ref{kinssect}).
Each matroid in the family has a number of circuit-hyperplanes,
and relaxing any one produces a non-representable matroid,
while relaxing two produces another matroid representable
over all infinite fields.
If there is a finite axiomatization of representability,
then that set of axioms must be able to distinguish between
these matroids.
Roughly speaking, we obtain a contradiction
by showing that, for large enough matroids in the family,
the number of circuit-hyperplanes is so great that an axiom
with a bounded number of variables cannot detect
all the potential relaxations.

\section{A language for matroids}
\label{logic}

In this section we develop monadic second-order
logic for matroids, and we describe \PH\ as
a set of formulas in MSOL.

\subsection{Monadic second-order logic}
Monadic second-order logic for matroids is a formal
language constructed from the following symbols:
the variables $x_{1},x_{2},x_{3},\ldots$ and
$X_{1},X_{2},X_{3},\ldots$;
the constants $\emptyset$, $0,1,2,\ldots$ and $E$;
the function symbols $|\cdot|$, $\{\cdot\}$, $\overline{\cdot}$,
$r(\cdot)$, $+$, $\cup$, and $\cap$;
the relation symbols $=$, $\in$, $\subseteq$, and $\leq$;
and the logical symbols $\neg$, $\lor$, $\land$, $\exists$, and
$\forall$.

\paragraph{\bf Terms.}
We divide the terms in MSOL into three classes,
\mcal{E}, \mcal{S}, and \mcal{N}.
Let \mcal{E} be the infinite set of
variables $\{x_{1},x_{2},x_{3},\ldots\}$.
The terms in \mcal{E} are intended to represent
elements of a ground set.

The set of terms in \mcal{S} is the smallest
collection of expressions satisfying:
\begin{enumerate}[label=(\arabic*)]
\item the constants $E$ and $\emptyset$ are terms
in \mcal{S},
\item every variable $X_{i}$ is a term in \mcal{S},
\item if $x_{i}$ is a variable in \mcal{E}, then
$\{x_{i}\}$ is a term in \mcal{S},
\item if $X$ and $Y$ are terms in \mcal{S}, then so are
$\overline{X}$, $X\cup Y$, and $X\cap Y$.
\end{enumerate}
The terms in \mcal{S} are intended to represent
subsets of a ground set.

Finally, we define the terms in \mcal{N}.
These are intended to represent non-negative integers.
The set of terms in \mcal{N} is the smallest set
of expressions satisfying:
\begin{enumerate}[label=(\arabic*)]
\item every constant in $\{0,1,2,\ldots\}$
is a term in \mcal{N},
\item if $X$ is a term in \mcal{S}, then
$|X|$ and $r(X)$ are terms in \mcal{N},
\item if $p$ and $q$ are terms in \mcal{N}, then
$p+q$ is a term in \mcal{N}.
\end{enumerate}

If $T$ is a term, then we recursively
define \var{T} to be the set of variables in $T$:
\begin{enumerate}[label=(\arabic*)]
\item \var{E} and \var{\emptyset} are empty, and so is \var{p},
for any constant $p\in \{0,1,2,\ldots\}$,
\item $\var{X_{i}}=\{X_{i}\}$,
\item $\var{x_{i}}=\var{\{x_{i}\}}=\{x_{i}\}$,
\item $\var{\overline{X}}=\var{|X|}=\var{r(X)}=\var{X}$, for any term
$X\in \mcal{S}$,
\item $\var{X\cup Y}=\var{X\cap Y}=\var{X}\cup \var{Y}$, for any
terms $X,Y\in \mcal{S}$,
\item $\var{p+q}=\var{p}\cup \var{q}$, for any terms
$p,q\in \mcal{N}$.
\end{enumerate}

\paragraph{\bf Formulas.}
Now we recursively define formulas in MSOL, and simultaneously
define their sets of variables.
The following four statements define
expressions known as \emph{atomic formulas}.
\begin{enumerate}[label=(\arabic*)]
\item if $x,y\in \mcal{E}$, then $x=y$ is an atomic formula,
and $\var{x=y}=\{x,y\}$.
\item if $X,Y\in \mcal{S}$, then $X=Y$ and
$X\subseteq Y$ are atomic formulas, and
$\var{X=Y}=\var{X\subseteq Y}=\var{X}\cup\var{Y}$,
\item if $p,q\in \mcal{N}$, then $p=q$ and $p\leq q$ are atomic formulas,
and $\var{p=q}=\var{p\leq q}=\var{p}\cup\var{q}$,
\item if $x\in \mcal{E}$ and $X\in\mcal{S}$, then
$x\in X$ is an atomic formula, and
$\var{x\in X}=\var{X}\cup \{x\}$,
\end{enumerate}

A \emph{formula} is an expression generated by a finite
application of the following rules.
Every formula has an associated set of variables
and \emph{free variables}:
\begin{enumerate}[label=(\arabic*)]
\item every atomic formula $P$ is a formula,
and $\fr{P}=\var{P}$,
\item if $P$ is a formula and $X_{i}\in \fr{P}$, then
$\exists X_{i} P$ and $\forall X_{i} P$ are formulas, and
$\var{\exists X_{i} P}=\var{\forall X_{i} P}=\var{P}$, while
$\fr{\exists X_{i} P}=\fr{\forall X_{i} P}=\fr{P}-\{X_{i}\}$,
\item if $P$ is a formula and $x_{i}\in\fr{P}$, then $\exists x_{i} P$
and $\forall x_{i} P$ are formulas, and
$\var{\exists x_{i} P}=\var{\forall x_{i} P}=\var{P}$, while
$\fr{\exists x_{i} P}=\fr{\forall x_{i} P}=\fr{P}-\{x_{i}\}$,
\item if $P$ is a formula, then
$\neg P$ is a formula, and $\var{\neg P}=\var{P}$, while
$\fr{\neg P}=\fr{P}$,
\item\label{union} if $P$ and $Q$ are formulas, and
$\fr{P}\cap (\var{Q}-\fr{Q})=\emptyset=(\var{P}-\fr{P})\cap\fr{Q}$,
then $P\lor Q$ and $P\land Q$ are formulas, and
$\var{P\lor Q}=\var{P\land Q}=\var{P}\cup\var{Q}$, while
$\fr{P\lor Q}=\fr{P\land Q}=\fr{P}\cup \fr{Q}$.
\end{enumerate}

A \emph{sentence} in MSOL is a formula $P$ 
satisfying $\fr{P}=\emptyset$.

\begin{remark}
In \ref{union}, when we construct the formulas $P\lor Q$
and $P\land Q$, we insist that no variable is free in
one of $P$ and $Q$ and not free in the other.
This is standard (see, for example, \cite[p.~10]{Mar02})
and imposes no real difficulties, since
a variable that is not free can always be relabeled.
For example, $(X_{1}=X_{2}) \land(\exists X_{1}\ |X_{1}|=1)$
is not a formula, but we can rewrite it as
 $(X_{1}=X_{2}) \land (\exists X_{3}\ |X_{3}|=1)$.
\end{remark}

\paragraph{\bf Abbreviations.}
We allow several standard shorthands.
If $P$ and $Q$ are formulas then $P\to Q$ is a shorthand for
$\neg P\lor Q$.
If $x\in \mcal{E}$ and $X\in \mcal{S}$, then
$x\notin X$ is shorthand for $\neg(x\in X)$.
If $p,q\in \mcal{N}$, then $p<q$ is shorthand for
$p\leq q\land \neg(p=q)$.
If $X,Y\in \mcal{S}$, then $X-Y$ is shorthand for
the term $X\cap \overline{Y}$,
and $X\nsubseteq Y$ is shorthand for the formula $\neg(X\subseteq Y)$.
In addition, we are casual with the use of parentheses,
inserting them freely to reduce ambiguity, and
omitting them when this will
cause no confusion.

\subsection{Structures and satisfiability}

We have constructed MSOL as a collection of formally defined strings.
In this section we are going to consider how to
interpret these strings as statements about a set system.
A \emph{structure}, \mcal{M}, consists of a
pair $(E^{\mcal{M}},r^{\mcal{M}})$,
where $E^{\mcal{M}}$ is a finite set and
$r^{\mcal{M}}$ is a function from $\mcal{P}(E^{\mcal{M}})$,
the power set of $E^{\mcal{M}}$, to the non-negative integers.

Let $\mcal{M}=(E^{\mcal{M}},r^{\mcal{M}})$ be
a structure, and let $P$ be a formula in MSOL.
Let $\phi_{\mcal{S}}$ be a function from
$\fr{P}\cap\mcal{S}$ to $\mcal{P}(E^{\mcal{M}})$
and let $\phi_{\mcal{E}}$ be a function from
$\fr{P}\cap\mcal{E}$ to $E^{\mcal{M}}$.
We call the pair $(\phi_{\mcal{S}},\phi_{\mcal{E}})$ an
\emph{interpretation} of $P$.
Note that an interpretation of a sentence necessarily
consists of two empty functions.
We are going to
recursively define what it means for the structure
\mcal{M} to satisfy $P(\phi_{\mcal{S}},\phi_{\mcal{E}})$.

First, we create a correspondence between terms in
\mcal{E}, \mcal{S}, and \mcal{N},
and elements of $E^{\mcal{M}}$, subsets of $E^{\mcal{M}}$,
and non-negative integers.
If $x_{i}$ is a term in \mcal{E}, and $x_{i}$ is in the domain of
$\phi_{\mcal{E}}$, then the notation\footnote{
Technically, we should write $x_{i}^{(\mcal{M},\phi_{\mcal{E}})}$,
since the element corresponding to $x_{i}$ depends
on the interpretation as well as the structure.
}
$x_{i}^{\mcal{M}}$ stands for
$\phi_{\mcal{E}}(x_{i})$.
Similarly, if $X$ is a term in \mcal{S}, and
$\var{X}\subseteq\operatorname{Dom}(\phi_{\mcal{S}})\cup
\operatorname{Dom}(\phi_{\mcal{E}})$,
then $X^{\mcal{M}}$ is the corresponding subset of $E^{\mcal{M}}$,
recursively defined as follows:
\begin{enumerate}[label=(\arabic*)]
\item if $X=E$, then $X^{\mcal{M}}=E^{\mcal{M}}$,
and if $X=\emptyset$, then $X^{\mcal{M}}$ is the empty subset,
\item if $X$ is the variable $X_{i}$, then $X_{i}^{\mcal{M}}=\phi_{\mcal{S}}(X_{i})$,
\item if $X=\{x_{i}\}$ for some variable $x_{i}$, then
$X^{\mcal{M}}=\{\phi_{\mcal{E}}(x_{i})\}$,
\item if $X=\overline{Y}$, for some $Y\in \mcal{S}$, then
$X^{\mcal{M}}=E^{\mcal{M}}-Y^{\mcal{M}}$,
and if $X$ is equal, respectively, to $Y\cup Z$ or
$Y\cap Z$, where $Y,Z\in\mcal{S}$, then
$X^{\mcal{M}}$ is, respectively,
$Y^{\mcal{M}}\cup Z^{\mcal{M}}$ or
$Y^{\mcal{M}}\cap Z^{\mcal{M}}$.
\end{enumerate}
Now let $p$ be a term in \mcal{N} such that
$\var{p}\subseteq\operatorname{Dom}(\phi_{\mcal{S}})\cup
\operatorname{Dom}(\phi_{\mcal{E}})$.
Then $p^{\mcal{M}}$ is the corresponding non-negative
integer, defined as follows:
\begin{enumerate}[label=(\arabic*)]
\item if $p$ is a constant in \mcal{N}, then $p^{\mcal{M}}$ is the
corresponding non-negative integer,
\item if $p$ is $|X|$ or $r(X)$, where $X$ is a term in
\mcal{S}, then $p^{\mcal{M}}$ is, respectively,
$|X^{\mcal{M}}|$, or $r^{\mcal{M}}(X^{\mcal{M}})$,
\item if $p$ is $q+r$, for some terms
$q,r\in \mcal{N}$, then $p^{\mcal{M}}$ is
$q^{\mcal{M}}+r^{\mcal{M}}$.
\end{enumerate}

Now we are able to recursively define when \mcal{M}
\emph{satisfies} $P(\phi_{\mcal{S}},\phi_{\mcal{E}})$.
First we consider the case that $P$ is an atomic formula:
\begin{enumerate}[label=(\arabic*)]
\item if $P$ is $x=y$, then $P(\phi_{\mcal{S}},\phi_{\mcal{E}})$
is satisfied if $x^{\mcal{M}}=y^{\mcal{M}}$,
\item if $P$ is, respectively, $X=Y$ or $X\subseteq Y$, then
$P(\phi_{\mcal{S}},\phi_{\mcal{E}})$ is satisfied if, respectively,
$X^{\mcal{M}}=Y^{\mcal{M}}$ or $X^{\mcal{M}}\subseteq Y^{\mcal{M}}$,
\item if $P$ is, respectively, $p=q$ or $p\leq q$, then
$P(\phi_{\mcal{S}},\phi_{\mcal{E}})$ is
satisfied if, respectively, $p^{\mcal{M}}=q^{\mcal{M}}$ or
$p^{\mcal{M}}\leq q^{\mcal{M}}$,
\item if $P$ is $x\in X$, then $P(\phi_{\mcal{S}},\phi_{\mcal{E}})$
is satisfied if $x^{\mcal{M}}\in X^{\mcal{M}}$.
\end{enumerate}

Next we consider the case that $P$ is not atomic:
\begin{enumerate}[label=(\arabic*)]
\item if $P=\exists X_{i} Q$, then $P(\phi_{\mcal{S}},\phi_{\mcal{E}})$
is satisfied if there is some subset $X_{i}'\subseteq E^{\mcal{M}}$
such that $Q(\phi_{\mcal{S}}\cup(X_{i},X_{i}'),\phi_{\mcal{E}})$
is satisfied;
and if $P=\forall X_{i} Q$, then $P(\phi_{\mcal{S}},\phi_{\mcal{E}})$
is satisfied if 
$Q(\phi_{\mcal{S}}\cup(X_{i},X_{i}'),\phi_{\mcal{E}})$
is satisfied for every subset
$X_{i}'\subseteq E^{\mcal{M}}$,
\item if $P=\exists x_{i} Q$, then $P(\phi_{\mcal{S}},\phi_{\mcal{E}})$
is satisfied if there is some element $x_{i}'\in E^{\mcal{M}}$
such that $Q(\phi_{\mcal{S}},\phi_{\mcal{E}}\cup(x_{i},x_{i}'))$
is satisfied;
and if $P=\forall x_{i} Q$, then $P(\phi_{\mcal{S}},\phi_{\mcal{E}})$
is satisfied if 
$Q(\phi_{\mcal{S}},\phi_{\mcal{E}}\cup(x_{i},x_{i}'))$
is satisfied for every element
$x_{i}'\in E^{\mcal{M}}$,
\item if $P=\neg Q$ is a formula, then
$P(\phi_{\mcal{S}},\phi_{\mcal{E}})$ is satisfied if
$Q(\phi_{\mcal{S}},\phi_{\mcal{E}})$ is not satisfied,
\item if $P=Q\lor R$, then $P(\phi_{\mcal{S}},\phi_{\mcal{E}})$
is satisfied if either
$Q(\phi_{\mcal{S}}|_{\fr{Q}\cap \mcal{S}},\phi_{\mcal{E}}|_{\fr{Q}\cap \mcal{E}})$
or
$R(\phi_{\mcal{S}}|_{\fr{R}\cap \mcal{S}},\phi_{\mcal{E}}|_{\fr{R}\cap \mcal{E}})$
is satisfied; and if $P=Q\land R$, then $P(\phi_{\mcal{S}},\phi_{\mcal{E}})$
is satisfied if both
$Q(\phi_{\mcal{S}}|_{\fr{Q}\cap \mcal{S}},\phi_{\mcal{E}}|_{\fr{Q}\cap \mcal{E}})$
and
$R(\phi_{\mcal{S}}|_{\fr{R}\cap \mcal{S}},\phi_{\mcal{E}}|_{\fr{R}\cap \mcal{E}})$
are satisfied.
\end{enumerate}

Let \mcal{M} be a structure, and let
$P$ be a sentence in MSOL.
We say that \mcal{M} satisfies $P$ if it satisfies
$P(\emptyset,\emptyset)$; that is, if it satisfies
$P$ under the empty interpretation.
If \mcal{T} is a set of sentences, then
\mcal{M} \emph{satisfies} \mcal{T} if it satisfies every sentence
in \mcal{T}.

\subsection{\PH}
\label{Mlogic}
Now we describe \PH\ as a set of formulas from MSOL.
Let $a$ be a variable.
Note that $\neg \exists a P$ is equivalent to
$\forall a \neg P$, in the sense that
a structure satisfies one of these formulas if and only if it
satisfies both.
Similarly, $\neg \forall a P$ is equivalent to
$\exists a \neg P$.
Now suppose that $P\lor (\exists a Q)$
is a formula.
Then $a$ is not free in $P$, and
$P\lor (\exists a Q)$ is equivalent to
$\exists a (P\lor Q)$.
Similarly, $P\land (\forall a Q)$ is equivalent to
$\forall a (P\land Q)$.
This discussion means that every formula in
MSOL is equivalent to a formula of the form
$Q_{1}a_{1}\cdots Q_{t}a_{t} P$, where
each $Q_{i}$ is in $\{\exists,\forall\}$, each
$a_{i}$ is a variable, and $P$ is a formula that
contains no quantifiers.

A formula of the form
$\exists x Q_{1}a_{1}\cdots Q_{t}a_{t} P$, where $x$ is a variable
in \mcal{E}, is equivalent to
\[
\exists X Q_{1}a_{1}\cdots Q_{t}a_{t}\forall x (X=\{x\})\to P
\]
where $X$ is a new variable in \mcal{S}.
Similarly, $\forall x Q_{1}a_{1}\cdots Q_{t}a_{t} P$ is equivalent to
$\forall X Q_{1}a_{1}\cdots Q_{t}a_{t}\forall x  (X=\{x\})\to P$.
From this discussion we see that every formula in MSOL
is equivalent to a formula of the form
\[
Q_{i_{1}}X_{i_{1}}\cdots Q_{i_{m}}X_{i_{m}}
Q_{j_{1}}x_{j_{1}}\cdots Q_{j_{n}}x_{j_{n}}P
\]
where $X_{i_{1}},\ldots, X_{i_{m}}$ and
$x_{j_{1}},\ldots, x_{j_{n}}$ are variables in
\mcal{S} and \mcal{E} respectively, where each $Q_{k}$ is
in $\{\exists,\forall\}$, and $\var{P}=\fr{P}$
(c.f.\ \cite[p.~39]{EF06}).
We say that this formula is in \emph{\PH} if
$\{Q_{i_{1}},\ldots, Q_{i_{m}}\}$ is either
$\{\exists\}$ or $\{\forall\}$, and similarly
$\{Q_{j_{1}},\ldots, Q_{j_{n}}\}$ is either
$\{\exists\}$ or $\{\forall\}$.
That is, \PH\ is the collection of formulas in MSOL that
are equivalent to a formula of the form
$Q_{i_{1}}X_{i_{1}}\cdots Q_{i_{m}}X_{i_{m}}
Q_{j_{1}}x_{j_{1}}\cdots Q_{j_{n}}x_{j_{n}}P$,
where $P$ is quantifier-free, and
$Q_{k}=Q_{l}$ for all $k,l\in\{i_{1},\ldots, i_{m}\}$
and all $k,l\in \{j_{1},\ldots, j_{n}\}$.

\section{Matroid axioms}
\label{interpret}

In this section we show that \PH\ is expressive enough
to make natural statements about matroids.
Some common axiom schemes for matroids can be expressed using sentences
in \PH.
Furthermore, if $N$ is a fixed matroid, then there is a
sentence in \PH\ that characterizes having a minor isomorphic to $N$.
Throughout the section, we will let
$\mcal{M}=(E^{\mcal{M}},r^{\mcal{M}})$ be a structure
(recall this implies $E^{\mcal{M}}$ is finite).

\subsection{Axioms}
\label{axioms}
We consider a matroid to be a finite set equipped with
a function obeying the rank axioms.
Thus $(E^{\mcal{M}},r^{\mcal{M}})$ is a matroid if and only if
\mcal{M} satisfies the following
sentences in \PH.
\begin{enumerate}[label={\bf R\arabic*}]
\item
$\forall X_{1}\ r(X_{1})\leq |X_{1}|$
\item
$\forall X_{1}\forall X_{2}\ 
X_{1}\subseteq X_{2} \rightarrow r(X_{1})\leq r(X_{2})$
\item
$\forall X_{1} \forall X_{2}\ 
r(X_{1}\cup X_{2})+r(X_{1}\cap X_{2})\leq r(X_{1})+r(X_{2})$
\end{enumerate}

Let $I(X)$ be shorthand for
$r(X)=|X|$, where $X\in \mcal{S}$.
Then $(E^{\mcal{M}},r^{\mcal{M}})$ is a matroid with
$\{X\subseteq E^{\mcal{M}}\mid r^{\mcal{M}}(X)=|X|\}$
as its family of independent sets if and only if
\mcal{M} satisfies the following sentences.
\begin{enumerate}[label={\bf I\arabic*}]
\item $I(\emptyset)$
\item $\forall X_{1}\forall X_{2}\ I(X_{2})\land X_{1}\subseteq X_{2}
\to I(X_{1})$
\item
$\forall X_{1}\forall X_{2}\exists x_{1}\ 
I(X_{1})\land I(X_{2})\land |X_{1}|<|X_{2}|\to$\\
\rule{0pt}{0pt}\hfill
$x_{1}\notin X_{1}\land x_{1}\in X_{2}\land  I(X_{1}\cup \{x_{1}\})$
\end{enumerate}

Let $B(X)$ be shorthand for $r(X)=|X|\land r(X)=r(E)$,
where $X\in \mcal{S}$.
Then $(E^{\mcal{M}},r^{\mcal{M}})$ is a matroid with
$\{X\subseteq E^{\mcal{M}}\mid r^{\mcal{M}}(X)=|X|=r^{\mcal{M}}(E^{\mcal{M}})\}$
as its family of bases if and only if
\mcal{M} satisfies the following sentences.
\begin{enumerate}[label={\bf B\arabic*}]
\item $\exists X_{1}\ B(X_{1})$
\item $\forall X_{1}\forall X_{2}\forall X_{3}\exists x_{1}\ 
B(X_{1})\land B(X_{2})\land |X_{3}|=1\land
X_{3}\subseteq X_{1}\land$\\
\rule{0pt}{0pt}\hfill
$X_{3}\nsubseteq X_{2} \to
x_{1}\notin X_{1}\land x_{1}\in X_{2}\land
B((X_{1}-X_{3})\cup \{x_{1}\})$
\end{enumerate}
Note that the natural form of the basis-exchange axiom
is
\begin{quote}
`for every basis $B$, and for every basis $B'$,
and for every element $x\in B-B'$, there exists
an element $y\in B'-B$ such that $\ldots$'
\end{quote}
This statement cannot be expressed directly in \PH.
We sidestep this problem by using the set variable, $X_{3}$,
to represent the single element $x$.

Let $S(X)$ be shorthand for $r(X)=r(E)$.
Then $(E^{\mcal{M}},r^{\mcal{M}})$ is a matroid with
$\{X\subseteq E^{\mcal{M}}\mid r^{\mcal{M}}(X)=r^{\mcal{M}}(E^{\mcal{M}})\}$
as its set of spanning sets if and only if
\mcal{M} satisfies the following sentences.
\begin{enumerate}[label={\bf S\arabic*}]
\item $\exists X_{1}\ S(X_{1})$
\item $\forall X_{1}\forall X_{2}\ S(X_{1})\land X_{1}\subseteq X_{2}
\to S(X_{2})$
\item $\forall X_{1}\forall X_{2}\exists x_{1}\ 
S(X_{1})\land S(X_{2})\land |X_{1}|<|X_{2}|\to $\\
\rule{0pt}{0pt}\hfill
$x_{1}\notin X_{1}\land x_{1}\in X_{2}\land
S(X_{2}-\{x_{1}\})$
\end{enumerate}

\subsection{Axiomatizing $\mathrm{GF}(q)$-representability}
\PH\ is strong enough so that representability over any
finite field can be axiomatized with a finite number of sentences,
assuming that Rota's conjecture is true.
This assumption implies that there is a finite number
of excluded minors for \gf{q}\dash representability, for
any prime power $q$.
In this section we prove the following result.

\begin{lemma}
\label{Rota}
Assume that Rota's conjecture is true.
For every finite field \gf{q}, there is a finite set of
sentences, \mcal{Q}, in \PH, with the following property:
the structure $\mcal{M}=(E^{\mcal{M}},r^{\mcal{M}})$
is a \gf{q}\dash representable matroid if and only if \mcal{M} satisfies
$\{{\bf R1},{\bf R2},{\bf R3}\}\cup \mcal{Q}$.
\end{lemma}

Indeed, any minor-closed class with finitely
many excluded minors can be finitely axiomatized
in \PH\ (Corollary~\ref{cor}).
However, the converse is not obviously true.
There may be a minor-closed class with infinitely
many excluded minors that can be finitely
axiomatized in \PH.

Lemma~\ref{Rota} follows
immediately from the next two results.

\begin{proposition}
\label{1}
Let $N$ be a matroid.
There is a sentence, $\mbf{S}_{N}$, in \PH, such that
the structure
$\mcal{M}=(E^{\mcal{M}},r^{\mcal{M}})$ is a matroid with an $N$\dash minor
if and only if \mcal{M} satisfies
$\{{\bf R1},{\bf R2},{\bf R3},\mbf{S}_{N}\}$.
\end{proposition}

\begin{proof}
Let the ground set of $N$ be $T=\{1,\ldots, m\}$.
For every subset $S\subseteq T$, let $r_{N}(S)$ denote the rank
of $S$ in $N$.
Let $P_{N}$ be the formula
\begin{multline*}
(r(X_{1})=|X_{1}|) \land
\left(X_{1}\cap \bigcup_{i=1}^{m}\{x_{i}\} =\emptyset\right) \land
\left(\left|\bigcup_{i=1}^{m}\{x_{i}\}\right|=m\right) \land\\
\bigwedge_{S\subseteq T}
r\left(X_{1}\cup \bigcup_{i \in S} \{x_{i}\}\right)
=r(X_{1})+r_{N}(S).
\end{multline*}
Assume that \mcal{M} satisfies
$\{{\bf R1},{\bf R2},{\bf R3}\}$, so that $(E^{\mcal{M}},r^{\mcal{M}})$
is a matroid.
Then \mcal{M} satisfies $P_{N}(\phi_{\mcal{S}},\phi_{\mcal{E}})$
if and only if $\phi_{\mcal{S}}(X_{1})$ is
independent, the set
$\{\phi_{\mcal{E}}(x_{1}),\ldots, \phi_{\mcal{E}}(x_{m})\}$
contains $m$ distinct elements and is disjoint from
$\phi_{\mcal{S}}(X_{1})$, and the matroid produced by
contracting $\phi_{\mcal{S}}(X_{1})$ and restricting to
$\{\phi_{\mcal{E}}(x_{1}),\ldots, \phi_{\mcal{E}}(x_{m})\}$
has the same rank function as $N$.
Thus $\mbf{S}_{N}=\exists X_{1}\exists x_{1}\cdots \exists x_{m}\ P_{N}$
is the desired sentence.
\end{proof}

\begin{corollary}
\label{cor}
If \mcal{N} is a minor-closed class of matroids with
a finite number of excluded minors, then there is a
finite set of sentences, $\mbf{S}(\mcal{N})$, in
\PH, with the following property:
the structure $\mcal{M}=(E^{\mcal{M}},r^{\mcal{M}})$ is a
matroid in \mcal{N} if and only if
\mcal{M} satisfies
$\{{\bf R1},{\bf R2},{\bf R3}\}\cup \mbf{S}(\mcal{N})$.
\end{corollary}

\begin{proof}
Let $N_{1},\ldots, N_{t}$ be the list of excluded
minors for \mcal{M}.
Notice that the negation of a sentence in \PH\
is equivalent to another sentence in \PH.
We let
$
\mbf{S}(\mcal{N})=\{\neg\mbf{S}_{N_{1}},\ldots,
\neg\mbf{S}_{N_{t}}\}.
$
\end{proof}

\section{Kinser matroids}
\label{kinssect}

In this section we construct an infinite family of matroids,
which we call \emph{Kinser matroids}.
Let $r \geq 4$ be an integer.
Then \kin{r} is a rank\dash $r$ matroid
with $r^{2}-3r+4$ elements.
For our purposes, the most important property of Kinser matroids
is that they are representable over any infinite
field, but can be made non-representable
by relaxing a single circuit-hyperplane.
To prove this fact, we are going to use the family
of inequalities discovered by Kinser~\cite{Kin11}.

\begin{lemma}
\label{kinser}
Let $M$ be a matroid that is representable over a field.
If $X_{1},\ldots, X_{n}$ is any collection of subsets of $E(M)$,
where $n\geq 4$, then
\begin{multline*}
r(X_{1}\cup X_{2})+r(X_{1}\cup X_{3}\cup X_{n})+r(X_{3})
+\sum_{i=4}^{n}(r(X_{i})+r(X_{2}\cup X_{i-1}\cup X_{i}))\leq\\
r(X_{1}\cup X_{3})+r(X_{1}\cup X_{n})+r(X_{2}\cup X_{3})+
\sum_{i=4}^{n}(r(X_{2}\cup X_{i})+r(X_{i-1}\cup X_{i})).
\end{multline*}
\end{lemma}

\begin{proof}
This follows immediately from
\cite[Theorem~1]{Kin11}.
\end{proof}

We note here that if $n=4$, then the inequality in
Lemma~\ref{kinser} is identical to Ingleton's
inequality for representable matroids \cite{Ing71}.

As an intermediate step for constructing \kin{r}, we define
a rank\dash $(r+1)$ transversal matroid, $M_{r+1}$.
The transversal system that describes $M_{r+1}$ contains
$r+1$ sets: $\mcal{A}_{1},\ldots, \mcal{A}_{r-1},\mcal{A},\mcal{A}'$.
Let $H_{1},\ldots, H_{r}$ be pairwise disjoint sets
such that
\[
|H_{1}|=\cdots=|H_{r-1}|=r-2
\]
and $H_{r}=\{e,f\}$.
The ground set of $M_{r+1}$ is
$H_{1}\cup\cdots\cup H_{r}$.
Let $\mcal{A}=E(M_{r+1})$, and let $\mcal{A}'=H_{r}$.
For $i\in \{1,\ldots, r-1\}$, let
\[
\mcal{A}_{i}=(H_{1}\cup\cdots\cup H_{r-1})-(H_{i-1}\cup H_{i})
\]
(when appropriate we interpret subscripts modulo $r-1$).
Then $M_{r+1}$ is the transversal matroid
$M[\mcal{A}_{1},\ldots, \mcal{A}_{r-1},\mcal{A},\mcal{A}']$.
We define \kin{r} to be the truncation, $T(M_{r+1})$, of
$M_{r+1}$.
Let $G_{r+1}$ be the bipartite graph that corresponds to the
transversal system
$(\mcal{A}_{1},\ldots, \mcal{A}_{r-1},\mcal{A},\mcal{A}')$.

Note that \kin{4} is a rank\dash $4$ matroid
with $8$ elements, and its non-spanning circuits
are all the $4$\dash element subsets of the form
$H_{i}\cup H_{j}$, where $i\ne j$.
In fact, \kin{4} is also known as the rank\dash $4$
\emph{tipless free spike} (see \cite[page~136]{GOVW02}).

We will use the next result in our proof that Kinser matroids
are representable over infinite fields.

\begin{proposition}
\label{infinity}
Let $r\geq 3$ be an integer.
Let $P$ be the projective geometry $\pg{r-1,\mbb{K}}$,
where \mbb{K} is an infinite field, and let
$S_{1},\ldots, S_{t}$ be a finite collection of proper
subspaces of $P$.
If $S$ is a subspace of $P$ that is not contained in
any of $S_{1},\ldots, S_{t}$, then $S$ is not contained in
$S_{1}\cup\cdots\cup S_{t}$.
\end{proposition}

\begin{proof}
Assume that the result is false, and that $S_{1},\ldots, S_{t}$
have been chosen so that none of these subsets contains
$S$, and yet $S_{1}\cup\cdots \cup S_{t}$ does.
Assume also that $S_{1},\ldots, S_{t}$ has been chosen so that
$t$ is as small as possible.
The hypotheses imply that $t>1$.
The minimality of $t$ means that there is a point, $p$, in
$S\cap S_{1}$, but not in
$S_{2}\cup\cdots\cup S_{t}$.
The same argument means that there is a point, $p'$,
contained in $(S\cap S_{t})-(S_{1}\cup \cdots\cup S_{t-1})$.
Note that $p\ne p'$ as $p\in S_{1}-S_{t}$ and
$p'\in S_{t}-S_{1}$.
Let $l$ be the line spanned by $p$ and $p'$.
Then $l$ is contained in $S$, but every subspace $S_{i}$
contains at most one point of $l$, for otherwise
$S_{i}$ contains $l$, and hence contains $p$ and $p'$.
Therefore $l$ contains at most $t$ points, contradicting
the fact that \mbb{K} is infinite.
\end{proof}

\begin{proposition}
\label{KinKrep}
Let \mbb{K} be an infinite field.
Then \kin{r} is \mbb{K}\dash representable for any $r \geq 4$.
\end{proposition}

\begin{proof}
Certainly $M_{r+1}$ is \mbb{K}\dash representable, as it
is transversal
(see \cite[Corollary~11.2.17]{Oxl11}).
Consider a \mbb{K}\dash representation of $M_{r+1}$
as an embedding of $E(M_{r+1})$ in the projective
space \pg{r,\mbb{K}}.

The non-spanning subsets of $M_{r+1}$ span a finite number of
proper subspaces of \pg{r,\mbb{K}}.
We let $S=\pg{r,\mbb{K}}$, and apply Proposition~\ref{infinity}.
Thus there is a point $p\in \pg{r,\mbb{K}}$ that 
is not spanned by any non-spanning subset
of $E(M_{r+1})$.
Consider the \mbb{K}\dash representable matroid, $M_{r+1}'$,
represented by the subset $E(M_{r+1})\cup p$ of \pg{r,\mbb{K}}.
Then $M_{r+1}'$ is a free extension of $M_{r+1}$; that is,
the only circuits that contain $p$ are spanning circuits.
Contracting $p$ produces the truncation $T(M_{r+1})=\kin{r}$.
Since $M_{r+1}'/p=\kin{r}$ is \mbb{K}\dash representable,
the proof is complete.
\end{proof}

\begin{proposition}
Let $r\geq 4$ be an integer.
Then $H_{s}\cup H_{r}$ is a circuit-hyperplane
of $\kin{r}$ for any $s\in \{1,\ldots, r-1\}$.
\end{proposition}

\begin{proof}
In $G_{r+1}$, the $r-2$ vertices in $H_{s}$ are each
adjacent to the $r-2$ vertices
\[
\{\mcal{A}_{1},\ldots, \mcal{A}_{r-1},\mcal{A}\}-
\{\mcal{A}_{s},\mcal{A}_{s+1}\},
\]
while the two vertices in $H_{r}$ are adjacent only to
$\mcal{A}$ and $\mcal{A}'$.
Thus $H_{s}\cup H_{r}$ contains $r$ vertices and
has a neighbourhood set of $r-1$ vertices.
Therefore $H_{s}\cup H_{r}$ is dependent, and in fact
it is very easy to confirm that it is a circuit of $M_{r+1}$.
Since it has cardinality $r$,
it is also a circuit in $T(M_{r+1})=\kin{r}$.

Let $x$ be an element in $E(M_{r+1})-(H_{s}\cup H_{r})$.
Then $x$ is adjacent to either $\mcal{A}_{s}$ or $\mcal{A}_{s+1}$
in $G_{r+1}$.
Thus the vertices in $H_{s}\cup H_{r}\cup x$ are adjacent to
$r$ vertices, so
\[r_{M_{r+1}}(H_{s}\cup H_{r}\cup x)>r_{M_{r+1}}(H_{s}\cup H_{r}).\]
This shows that  $H_{s}\cup H_{r}$ is a flat in $M_{r+1}$.
As $r_{M_{r+1}}(H_{s}\cup H_{r})=r-1=r(M_{r+1})-2$,
it follows that $H_{s}\cup H_{r}$ is also a flat in
$T(M_{r+1})=\kin{r}$, and is therefore a hyperplane of this matroid.
Thus $H_{s}\cup H_{r}$ is a circuit-hyperplane of
$\kin{r}$.
\end{proof}

\begin{proposition}
\label{Kin-}
Let $r\geq 4$ be an integer, and let $s$ be in $\{1,\ldots, r-1\}$.
The matroid obtained from \kin{r} by relaxing the
circuit-hyperplane $H_{s}\cup H_{r}$
is not representable over any field.
\end{proposition}

\begin{proof}
By relabeling $\mcal{A}_{i}$ and $H_{i}$ as
$\mcal{A}_{i-s+1}$ and $H_{i-s+1}$ (modulo $r-1$) for each
$i\in \{1,\ldots, r-1\}$,
we can assume that $s=1$.
Let $M$ be the matroid obtained from \kin{r} by relaxing
$H_{1}\cup H_{r}$.
We prove that $M$ is non-representable by setting $n=r$ and setting
\[
(X_{1},X_{2},X_{3},\ldots, X_{n})=(H_{1},H_{r},H_{2},\ldots, H_{r-1}),
\]
and then applying Lemma~\ref{kinser} to $M$.
Since $X_{1}\cup X_{2}$ is a relaxed circuit-hyperplane in
$M$, it follows that
$r(X_{1}\cup X_{2})=r$.
Note that $H_{i}\cup H_{r}$ is a circuit-hyperplane of
$M$ for any $i\in \{2,\ldots, r-1\}$.
Thus, any set $X_{i}$, where $i\in \{3,\ldots, n\}$,
is an $(r-2)$\dash element subset of a circuit-hyperplane.
This means that $r(X_{i})=r-2$.
In particular, $r(X_{3})=r(H_{2})=r-2$.
In the bipartite graph $G_{r+1}$, the vertices in $H_{2}$ are
adjacent to the $r-2$ vertices
$\mcal{A}_{1},\mcal{A}_{4},\ldots,\mcal{A}_{r-1},\mcal{A}$.
Each vertex in $H_{1}$ is adjacent to $\mcal{A}_{3}$, while
every vertex in $H_{r-1}$ is adjacent to $\mcal{A}_{2}$.
These considerations imply that $H_{1}\cup H_{2}\cup H_{r-1}$ has
rank at least $r$ in $M_{r+1}$, and hence in $M$.
Thus $r(X_{1}\cup X_{3}\cup X_{n})=r$.
For $i\in\{4,\ldots, n\}$, the set
$X_{2}\cup X_{i-1}=H_{r}\cup H_{i-2}$ is a circuit-hyperplane
of $M$.
It follows that
$r(X_{2}\cup X_{i-1}\cup X_{i})=r$.
Now the left-hand side of the inequality in Lemma~\ref{kinser}
evaluates to
\[
r+r+(r-2)+(r-3)[(r-2)+r]
=2r^{2}-5r+4.
\]

On the other hand, if $i\in\{1,\ldots, r-1\}$, then the
neighbourhood in $G$ of $H_{i}\cup H_{i+1}$ contains
the $r-1$ vertices
$\{\mcal{A}_{1},\ldots, \mcal{A}_{r-1},\mcal{A}\}-\mcal{A}_{i+1}$.
Thus $H_{i}\cup H_{i+1}$ has rank at most $r-1$ in $M_{r+1}$.
In fact it has rank exactly $r-1$, as $H_{i}$ as rank $r-2$,
and any vertex in $H_{i+1}$ is adjacent to $\mcal{A}_{i}$,
while no vertex in $H_{i}$ is.
Thus $H_{i}\cup H_{i+1}$ has rank $r-1$ in $M$.
This shows that
$r(X_{1}\cup X_{3})$,
$r(X_{1}\cup X_{n})$, and
$r(X_{i-1}\cup X_{i})$ for $i\in \{4,\ldots, n\}$ are
all equal to $r-1$.
Furthermore, $X_{2}\cup X_{i}$ is a circuit-hyperplane
for all $i\in \{3,\ldots, n\}$, so has rank $r-1$.
Now every term in the right-hand side of the inequality in
Lemma~\ref{kinser} is equal to $r-1$, so this side
evaluates to $(2(r-3)+3)(r-1)=2r^{2}-5r+3$.
Thus the inequality in Lemma~\ref{kinser} does not hold, so
$M$ is not representable over any field.
\end{proof}

If $r\geq 4$ is an integer, then we define $\kin{r}^{-}$
to be the matroid obtained from \kin{r} by relaxing
the circuit-hyperplane $H_{1}\cup H_{r}$.
The previous result shows that $\kin{r}^{-}$ is
non-representable.
Since \kin{4} is isomorphic to the 
rank\dash $4$ tipless free spike, it is easy to see that
$\kin{4}^{-}$ is the \emph{V\'{a}mos matroid}
(see \cite[page~84]{Oxl11} or \cite{Vam68}).
In fact, we can think of $\kin{n}^{-}$ as exemplifying
matroids that fail the inequality in Lemma~\ref{kinser},
in exactly the same way that the V\'{a}mos matroid exemplifies
matroids that fail the Ingleton inequality \cite{Ing71}.

Relaxing a single circuit-hyperplane in \kin{r}
produces a non-representable matroid.
We show in the next result that by relaxing two,
we can recover representability over any infinite field.

\begin{lemma}
\label{doublerelax}
Let \mbb{K} be an infinite field,
let $r\geq 4$ be an integer, and let $s$ and $t$ be distinct
members of $\{1,\ldots, r-1\}$.
The matroid that is obtained from \kin{r} by relaxing
the circuit-hyperplanes $H_{s}\cup H_{r}$ and
$H_{t}\cup H_{r}$ is \mbb{K}\dash representable.
\end{lemma}

\begin{proof}
We assume that $s<t$.
By relabeling $\mcal{A}_{i}$ and $H_{i}$ as
$\mcal{A}_{i-t+r-1}$ and $H_{i-t+r-1}$ for
every $i\in\{1,\ldots, r-1\}$, we can assume that
$t=r-1$.
Relabel $s-t+r-1$ as $s$.
Let $M$ be the matroid obtained from \kin{r} by relaxing
$H_{s}\cup H_{r}$ and $H_{r-1}\cup H_{r}$.
We aim to show that $M$ is \mbb{K}\dash representable.

We start by constructing a rank\dash $r$ transversal matroid, $M'$,
on the ground set
\[E(M_{r+1}\ba \{e,f\})\cup\{p,p'\},\]
where $p$ and $p'$ are distinct elements, 
neither of which is in $E(M_{r+1})$.
Let $\mcal{A}_{0}$ be
$E(M_{r+1}\ba \{e,f\})\cup \{p,p'\}$.
For $i\in \{1,\ldots, s\}$, let $\mcal{A}_{i}'$ be
$\mcal{A}_{i}\cup p$.
For $i\in \{s+1,\ldots,r-1\}$, let
$\mcal{A}_{i}'$ be $\mcal{A}_{i}\cup p'$.
Let $M'$ be the transversal matroid
$M[\mcal{A}_{1}',\ldots, \mcal{A}_{r-1}',\mcal{A}_{0}]$.

It is clear that
$M'\ba \{p,p'\}=M_{r+1}\ba \{e,f\}$.
Moreover, it is straightforward to verify that $\{e,f\}$ is a series pair
in $M_{r+1}$, and from this it follows easily that
\[M_{r+1}\ba \{e,f\}=
T(M_{r+1})\ba \{e,f\}=\kin{r}\ba \{e,f\}.
\]
Thus $M'\ba \{p,p'\}=\kin{r}\ba\{e,f\}$.

Since $M'$ is transversal, it is \mbb{K}\dash representable.
We consider it as a subset of points in the projective space
$P=\pg{r-1,\mbb{K}}$.
Let $l$ be the line of $P$ that is spanned by
$p$ and $p'$.

\begin{claim}
Let $X$ be a subset of
$E(M_{r+1}\ba \{e,f\})$ that is
non-spanning in $M'$.
Then $l\nsubseteq \cl_{P}(X)$.
\end{claim}

\begin{proof}
Assume otherwise.
Then there is a subset of $E(M_{r+1}\ba\{e,f\})$ that
spans $l$ and is independent and non-spanning in $M'$.
Let $X$ be such a subset.
Thus $p,p'\in\cl_{M'}(X)$.
Let $C\subseteq X\cup p$ be a circuit of $M'$ that contains $p$.
Let $c$ be an element in $C-p$.
Then in $G'$, the bipartite graph corresponding to the
system $(\mcal{A}_{1}',\ldots, \mcal{A}_{r-1}',\mcal{A}_{0})$,
the vertex $c$ has $r-2$ neighbours.
Since the neighbourhood set of $C$ is one element smaller
than $C$, this means that $|C|\geq r-1$.
Let $C'\subseteq X\cup p'$ be a circuit that contains $p'$.
The same argument shows that $|C'|\geq r-1$.
Since
\[r>|X|\geq|(C-p)\cup (C'-p')|\geq (2r-4)-|(C-p)\cap(C'-p')|\]
and $r\geq 4$, this means that there is an element,
$x$, in $(C-p)\cap (C'-p')$.
Assume that $x$ is in one of $H_{1},\ldots, H_{s-1}$.
As $p$ is adjacent to 
the vertices $\mcal{A}_{0},\mcal{A}_{1}',\ldots, \mcal{A}_{s}'$,
and $x$ is adjacent to all vertices, other than two
in $\mcal{A}_{1}',\ldots, \mcal{A}_{s}'$,
it follows that
the neighbourhood set of $C$ contains $r$ vertices.
This means that $|C|\geq r+1$, which is impossible as
$X$ is non-spanning.
Similarly, if $x$ is in one of $H_{s+1},\ldots, H_{r-2}$,
then, as $p'$ is adjacent to
$\mcal{A}_{s+1}',\ldots, \mcal{A}_{r-1}'$, and 
$x$ is adjacent to every vertex other than two in
$\mcal{A}_{s+1}',\ldots, \mcal{A}_{r-1}'$, we deduce that
$|C'|\geq r+1$.
This contradiction means that
$(C-p)\cap(C'-p)$ is contained in $H_{s}\cup H_{r-1}$.
If $(C-p)\cap(C'-p')$ contains elements from both $H_{s}$ and
$H_{r-1}$, then the neighbourhood set of either $C$ or $C'$
contains all $r$ vertices, and this leads to the same contradiction
as before.
Thus $(C-p)\cap (C'-p')$ is contained in either
$H_{s}$ or $H_{r-1}$.
Thus the neighbourhood set of $C$ includes every vertex
other than either $\mcal{A}_{s+1}'$ or $\mcal{A}_{r-1}'$,
meaning that $|C|\geq r$, and hence $|C|=r$.
Similarly, the neighbourhood set of $C'$ contains every vertex
other than either $\mcal{A}_{1}'$ or $\mcal{A}_{s}'$, so
$|C'|=r$.
As $r>|X|$ and $X\supseteq (C-p)\cup (C'-p')$, we deduce that
$C-p=C'-p'$.
Our earlier arguments show that $C$ is contained in either
$H_{s}\cup p$ or $H_{r-1}\cup p$.
But this means that $|C|\leq r-1$, and the neighbourhood
set of $C$ contains all of the $r$ vertices other than
either $\mcal{A}_{s+1}'$ or $\mcal{A}_{r-1}'$.
This contradicts the fact that $C$ is a circuit, and
completes the proof of the claim.
\end{proof}

Consider all the subspaces of $P$ that are spanned by
subsets of $E(M_{r+1}\ba \{e,f\})$ that are non-spanning in $M'$.
This is a finite collection of subspaces, and the previous
claim says that none of them contains $l$.
By Proposition~\ref{infinity}, there is a point, $f$, on
$l$ that is not spanned by any non-spanning subset of
$E(M_{r+1}\ba\{e,f\})$.
We can apply the same argument, augmenting the collection
of subspaces with $\langle \{f\}\rangle$, and find another,
distinct, point, $e$, on $l$ that is not spanned by any
non-spanning subset of $E(M_{r+1}\ba\{e,f\})$.
Consider the \mbb{K}\dash representable matroid corresponding to the
subset $H_{1}\cup\cdots\cup H_{r-1}\cup\{e,f\}$ of $P$.
Let this matroid be $N$.
We will show that $N=M$, and this
will complete the proof of Lemma~\ref{doublerelax}.

Certainly $N\ba\{e,f\}=M'\ba\{p,p'\}$, and we deduced earlier
that $M'\ba \{p,p'\}=\kin{r}\ba \{e,f\}$.
As $e$ and $f$ are contained in the
circuit-hyperplanes $H_{s}\cup H_{r}$ and
$H_{r-1}\cup H_{r}$, deleting them from $M$ effectively
undoes the relaxations that produced $M$
(see \cite[Proposition~3.3.5]{Oxl11});
that is, $\kin{r}\ba \{e,f\}=M\ba \{e,f\}$.
Now we have shown that $N\ba \{e,f\}=M\ba \{e,f\}$.
Moreover, in $N\ba e$, the element $f$ is freely
placed by construction, so $N\ba e$ is a free extension of
$N\ba \{e,f\}$.
On the other hand, as $e$ is in $H_{s}\cup H_{r}$ and
$H_{r-1}\cup H_{r}$, it follows that
\[M\ba e=\kin{r}\ba e=T(M_{r+1})\ba e=T(M_{r+1}\ba e).\]
But $f$ is a coloop in $M_{r+1}\ba e$, so it is freely placed
in $T(M_{r+1}\ba e)=M\ba e$.
Therefore $M\ba e$ is a free extension of $M\ba \{e,f\}$.
As $N\ba \{e,f\}=M\ba \{e,f\}$, this means that $N\ba e=M\ba e$.

Assume that $N\ne M$.
Then there is a set, $X$, which is a non-spanning circuit in
one of $\{M,N\}$, and independent in the other.
As $N\ba e=M\ba e$, it follows that $e$ is in $X$.
We will show that $f$ is also in $X$.
If $X$ is a non-spanning circuit of $N$, then $f\in X$,
for otherwise $X-e$ is a non-spanning subset of
$E(M_{r+1}\ba \{e,f\})$ that spans $e$, and $N$ was constructed
so that no such subset exists.
Therefore assume that $X$ is a non-spanning circuit in $M$.
Then $X$ is also a non-spanning circuit in $\kin{r}=T(M_{r+1})$,
and hence in $M_{r+1}$.
But $\{e,f\}$ is a series pair in $M_{r+1}$, so any circuit
that contains $e$ also contains $f$.
Thus $X$ contains $\{e,f\}$ in either case.

First we assume that $X$ is independent in $N$ and 
a non-spanning circuit of $M$, and hence of \kin{r} and
$M_{r+1}$.
Since $|X|\leq r$, the neighbourhood set of $X$ in $G_{r+1}$,
the bipartite graph corresponding to
$(\mcal{A}_{1},\ldots, \mcal{A}_{r-1},\mcal{A},\mcal{A}')$,
has at most $r-1$ vertices.
If $X-\{e,f\}$ contains elements from two distinct sets
in $\{H_{1},\ldots, H_{r-1}\}$, then the neighbourhood set
of these two elements contains all but at most one vertex
from $\{\mcal{A}_{1},\ldots, \mcal{A}_{r-1}\}$.
As $e$ and $f$ are adjacent to \mcal{A} and $\mcal{A}'$,
this means that $X$ has a neighbourhood set containing
$r$ vertices.
It follows that $X-\{e,f\}$ is contained in one of
$H_{1},\ldots, H_{r-1}$.
Thus the neighbourhood set of $X$ contains
exactly $r-1$ vertices.
Thus $X$ has cardinality $r$, so $X=H_{i}\cup\{e,f\}$, for
some $i\in \{1,\ldots, r-1\}$.
However, $i$ is not $s$ or $r-1$, as
$H_{s}\cup\{e,f\}$ and $H_{r-1}\cup\{e,f\}$ are bases in
$M$.
If $i\in \{1,\ldots, s-1\}$, then in the bipartite graph $G'$,
the $r-1$ vertices in $H_{i}\cup p'$ are adjacent to the
$r-2$ vertices in
\[\{\mcal{A}_{1}',\ldots, \mcal{A}_{r-1}',\mcal{A}_{0}\}
-\{\mcal{A}_{i}',\mcal{A}_{i+1}'\}.
\]
Thus $H_{i}\cup p'$ is dependent in $M'$.
But $\{e,f,p'\}$ is dependent in $P|E(M')\cup\{e,f\}$.
It follows easily that $H_{i}\cup\{e,f\}=X$ is
dependent in $P|E(M')\cup\{e,f\}$, and hence in $N$.
Similarly, if $i\in\{s+1,\ldots, r-2\}$, then the neighbourhood
set of $H_{i}\cup p$ is
\[\{\mcal{A}_{1}',\ldots, \mcal{A}_{r-1}',\mcal{A}_{0}\}
-\{\mcal{A}_{i}',\mcal{A}_{i+1}'\},
\]
so $H_{i}\cup p$ and $\{e,f,p\}$ are dependent.
This leads to the contradiction that $X$ is dependent in $N$.
Hence we now assume that $X$ is a non-spanning circuit of $N$.

Note that $X$ and $\{e,p,p'\}$ are both circuits of
$P|E(N)\cup\{p,p'\}$.
We apply strong circuit-elimination, and deduce that there is
a circuit, $C$, contained in $(X-e)\cup\{p,p'\}$ that
contains $p$.
If $f\in C$, then we apply strong circuit-elimination to
$C$ and $\{f,p,p'\}$, and find a circuit that contains $p$
but not $f$.
Thus we lose no generality in assuming that
$C\subseteq (X-\{e,f\})\cup\{p,p'\}$ is a circuit of $M'$ that
contains $p$.
If $p'$ is in $C$, then the neighbourhood set of $C$
in $G'$ contains all $r$ vertices
$\mcal{A}_{1}',\ldots, \mcal{A}_{r-1}',\mcal{A}_{0}$.
Thus $|X|\geq |C|\geq r+1$, which is impossible.
Hence $p'\notin C$.
If $C$ contains an element from $H_{1}\cup\cdots\cup H_{s}$ or
$H_{r-1}$, then the neighbourhood set of $C$ in $G'$ contains all but
at most one vertex from
$\mcal{A}_{1}',\ldots, \mcal{A}_{r-1}',\mcal{A}_{0}$.
Thus $|C|\geq r$.
As $p'\notin X$ implies $|X|\geq |C|+1$, this leads to a contradiction.
Therefore $C-p$ is contained in
$H_{s+1}\cup\cdots\cup H_{r-2}$.
If $C$ contains elements from two of $H_{s+1},\ldots, H_{r-2}$,
then its neighbourhood set again contains at least $r-1$
elements, leading to a contradiction.
Therefore $C-p$ is contained in one of $H_{s+1},\ldots, H_{r-2}$,
so the neighbourhood set of $C$ contains $r-2$ elements.
It follows that $|C|=r-1$, so $C=H_{i}\cup p$ for
some $i\in\{s+1,\ldots, r-2\}$.
As $|X|\leq r$, this implies that $X=H_{i}\cup\{e,f\}$.
But then $X$ is a circuit-hyperplane in $M$, contradicting the
fact that it is independent in this matroid.

We conclude that $N=M$, so $M$ is \mbb{K}\dash representable,
as desired.
\end{proof}

Recall that $\kin{r}^{-}$ is the matroid obtained from
\kin{r} by relaxing $H_{1}\cup H_{r}$.
If $r\in\{2,\ldots, r-1\}$, then we will let $\kin{r}_{i}^{=}$
be the matroid obtained from $\kin{r}^{-}$ by relaxing
$H_{i}\cup H_{r}$.
Thus the results in this section show that \kin{r} and
$\kin{r}_{i}^{=}$ are representable over any infinite field,
and that $\kin{r}^{-}$ is representable over no field.

\section{Proof of the main theorem}

In this section we prove our main theorem.
Theorem~\ref{main2} is a restatement of Theorem~\ref{main}
that uses slightly different language.
If \mcal{F} is a set of fields, then define
$M(\mcal{F})$ to be
\[
\bigcup_{F\in \mcal{F}}
\{M\mid M\ \text{is an}\ F\text{-representable matroid}\}.
\]

\begin{theorem}
\label{main2}
Let \mcal{F} be a set of fields that contains at least
one infinite field.
There does not exist a finite set, \mcal{K}, of sentences in \PH\
with the following property: if
$\mcal{M}=(E^{\mcal{M}},r^{\mcal{M}})$ is a structure,
then $(E^{\mcal{M}},r^{\mcal{M}})$ is a
matroid in $M(\mcal{F})$ if and only if \mcal{M} satisfies
$\{{\bf R1},{\bf R2},{\bf R3}\}\cup\mcal{K}$.
\end{theorem}

Before we prove this theorem, we discuss some preliminaries.
Assume that $\mcal{M}=(E^\mcal{M},r^{\mcal{M}})$ is a
structure.
For every function, $\phi$, into $\mcal{P}(E^{\mcal{M}})$, there is
an induced family of subsets of $E^{\mcal{M}}$ that we call
\emph{definable} subsets (relative to $\phi$).
Let us say that a \emph{minterm} is a subset of
$E^{\mcal{M}}$ that can be expressed in the form
\[\bigcap_{X\in\operatorname{Dom}(\phi)} f(X),\]
where $f(X)$ is either $\phi(X)$ or $E^{\mcal{M}}-\phi(X)$, and the
intersection ranges over the domain of $\phi$.
Note that distinct minterms are disjoint, and that every element
of $E^{\mcal{M}}$ is in a minterm.
We say that a subset of $E^{\mcal{M}}$ is \emph{definable} if
it is a union of minterms.
Note that if the domain of $\phi$ has size $m$, then there
are at most $2^{m}$ possible minterms, and hence
at most $2^{2^{m}}$ definable subsets.

Now assume that $\{X_{i}\}_{i\in I}$ and
$\{x_{j}\}_{j\in J}$ are sets of variables in
\mcal{S} and \mcal{E} respectively, and that
$\phi_{\mcal{S}}\colon\{X_{i}\}_{i\in I}\to \mcal{P}(E^{\mcal{M}})$ and
$\phi_{\mcal{E}}\colon \{x_{j}\}_{j\in J} \to E^{\mcal{M}}$
are assignments of set and element variables to subsets and
elements of $E^{\mcal{M}}$.
We say that a set is definable relative to
$(\phi_{\mcal{S}},\phi_{\mcal{E}})$ if it is definable
relative to the function that takes
$X_{i}$ to $\phi_{\mcal{S}}(X_{i})$ for every $i\in I$, and
$x_{j}$ to $\{\phi_{\mcal{E}}(x_{j})\}$ for every $j\in J$.

Let $P$ be a formula in \PH\ such that
$\var{P}=\fr{P}$.
Let $(\phi_{\mcal{S}},\phi_{\mcal{E}})$ be an interpretation of $P$.
Observe that any set $\phi_{\mcal{S}}(X_{i})$ is definable
relative to $(\phi_{\mcal{S}},\phi_{\mcal{E}})$,
since it is the union of all minterms in which
$f(X_{i})=\phi_{\mcal{S}}(X_{i})$.
Similarly, any set $\{\phi_{\mcal{E}}(x_{j})\}$ is definable.
Both $E^{\mcal{M}}$ and the empty set are definable (the
former as the union of all possible minterms, the latter as the
empty union).
Furthermore, if $X$ and $Y$ are definable sets, then
$E^{\mcal{M}}-X$, $X\cup Y$, and $X\cap Y$ are also definable.
It follows that, if $X\in \mcal{S}$ is a term that appears
in $P$, then $X^{\mcal{M}}$ is one of the
\[
2^{2^{|\var{P}|}}
\]
definable subsets of $E^{\mcal{M}}$.

\begin{proposition}
\label{symdif}
Let $P$ be a formula in \PH\ such that
$\var{P}=\fr{P}$.
Let $(\phi_{\mcal{S}},\phi_{\mcal{E}})$ be an interpretation
of $P$, and
let $T=\phi_{\mcal{E}}(\var{P}\cap\mcal{E})$ be the image of
$\phi_{\mcal{E}}$.
Every definable set relative to $(\phi_{\mcal{S}},\phi_{\mcal{E}})$
can be expressed in the form $(A-T)\cup B$, where $A$ is definable relative to
$\phi_{\mcal{S}}$, and $B$ is a subset of $T$.
\end{proposition}

\begin{proof}
Consider a minterm
\[
Z=\bigcap_{x_{j}\in\var{P}\cap\mcal{E}}f(x_{j}),
\]
relative to the function that takes every
variable $x_{j}$ to $\{\phi_{\mcal{E}}(x_{j})\}$.
If $f(x_{j_{1}})=\{\phi_{\mcal{E}}(x_{j_{1}})\}$ and
$f(x_{j_{2}})=\{\phi_{\mcal{E}}(x_{j_{2}})\}$, for variables
$x_{j_{1}}$ and $x_{j_{2}}$ such that
$\phi_{\mcal{E}}(x_{j_{1}})\ne\phi_{\mcal{E}}(x_{j_{2}})$,
then $Z=\emptyset$.
If all the variables $x_{j}$ satisfying
$f(x_{j})=\{\phi_{\mcal{E}}(x_{j})\}$, have the same
image under $\phi_{\mcal{E}}$, then $Z$ is either the
empty set, or a singleton subset of $T$.
Finally, if $f(x_{j})=E^{\mcal{M}}-\{\phi_{\mcal{E}}(x_{j})\}$
for every variable $x_{j}$, then $Z=E^{\mcal{M}}-T$.

Every minterm relative to
$(\phi_{\mcal{S}},\phi_{\mcal{E}})$ is the intersection
of a minterm relative to $\phi_{\mcal{S}}$ with a minterm
relative to the function $x_{j}\mapsto\{\phi_{\mcal{E}}(x_{j})\}$.
Thus every minterm relative to
$(\phi_{\mcal{S}},\phi_{\mcal{E}})$ is either
the empty set, a singleton subset of $T$,
or the intersection of $E^{\mcal{M}}-T$ with a
minterm relative to $\phi_{\mcal{S}}$.
Now it is clear that any union of such minterms is
the union of a subset of $T$, and the intersection
of $A$ with $E^{\mcal{M}}-T$, where
$A$ is a definable subset relative to $\phi_{\mcal{S}}$.
Thus the proposition holds.
\end{proof}

\begin{proof}[Proof of Theorem~{\rm\ref{main2}}.]
We assume for a contradiction that \mcal{K} is a finite set of
sentences in \PH\ having the property that
$\mcal{M}=(E^{\mcal{M}},r^{\mcal{M}})$ is
a matroid in $M(\mcal{F})$ if and only if \mcal{M} satisfies
$\{{\bf R1},{\bf R2},{\bf R3}\}\cup\mcal{K}$.

Let $L$ be an integer such that
$|\var{S}|\leq L$ for every sentence $S\in \mcal{K}$.
Let
\[N=2^{2^{L}}+3,\]
and let $E^{\mcal{M}}=E(\kin{N})$.
 
Since $\kin{N}^{-}$ is not representable,
by Proposition~\ref{Kin-}, there is a sentence
in \mcal{K} that is not satisfied when $r^{\mcal{M}}$ is the
rank function of $\kin{N}^{-}$.
Let $S$ be such a sentence.
We can assume $S$ is a formula with one of the following forms:
\begin{enumerate}[label=(\arabic*)]
\item
$\exists X_{i_{1}}\cdots \exists X_{i_{m}}
\exists x_{j_{1}}\cdots \exists x_{j_{n}}P$
\item
$\exists X_{i_{1}}\cdots \exists X_{i_{m}}
\forall x_{j_{1}}\cdots \forall x_{j_{n}}P$
\item
$\forall X_{i_{1}}\cdots \forall X_{i_{m}}
\forall x_{j_{1}}\cdots \forall x_{j_{n}}P$
\item
$\forall X_{i_{1}}\cdots \forall X_{i_{m}}
\exists x_{j_{1}}\cdots \exists x_{j_{n}}P$
\end{enumerate}
where $P$ is a formula such that
$\var{P}\cap \mcal{S}=\fr{P}\cap \mcal{S}=\{X_{i_{1}},\ldots,X_{i_{m}}\}$
and
$\var{P}\cap \mcal{E}=\fr{P}\cap \mcal{E}=\{x_{j_{1}},\ldots,x_{j_{n}}\}$.
Let $I$ and $J$ be the index sets $\{i_{1},\ldots, i_{m}\}$
and $\{j_{1},\ldots, j_{n}\}$.
Note that $m+n\leq L$.

\paragraph{\bf Case 1.}
We first assume that $S$ has the form
\[
\exists X_{i_{1}}\cdots\exists X_{i_{m}}
\exists x_{j_{1}}\cdots\exists x_{j_{n}}P.
\]
Since \kin{N} is representable over at least one field
in \mcal{F}
(Proposition~\ref{KinKrep}), there is an interpretation,
$(\phi_{\mcal{S}},\phi_{\mcal{E}})$,
such that $P(\phi_{\mcal{S}},\phi_{\mcal{E}})$ is satisfied when
$r^{\mcal{M}}$ is the rank function of \kin{N}.
Consider the definable subsets relative to
$(\phi_{\mcal{S}},\phi_{\mcal{E}})$.
There are at most $2^{2^{m+n}}\leq 2^{2^{L}}$ such subsets.
As
\[N-1=2^{2^{L}}+2\]
is greater than the number of definable
subsets, there is an index $s\in\{1,\ldots, N-1\}$ such that
$H_{s}\cup H_{N}$ is not definable.
Let $M$ be the matroid obtained from
\kin{N} by relaxing $H_{s}\cup H_{N}$.
The rank functions of $M$ and \kin{N} differ
only on the set $H_{s}\cup H_{N}$.
Since this set is not definable, we see that
if $X$ is any set term appearing in $P$,
then the rank of $X^{\mcal{M}}$ in $M$ is the same as its rank in
\kin{N}.
Thus $P(\phi_{\mcal{S}},\phi_{\mcal{E}})$ is satisfied when
$r^{\mcal{M}}$ is the rank function of $M$.
For $k\in\{1,\ldots, N-1\}$, let $p_{k}$ be an arbitrary
bijection from $H_{k}$ to $H_{k-s+1}$.
These bijections clearly induce an isomorphism from $M$ to
$\kin{N}^{-}$.
By composing this isomorphism with $\phi_{\mcal{S}}$ and
$\phi_{\mcal{E}}$, we obtain an interpretation that satisfies
$P$ when $r^{\mcal{M}}$ is the rank function of $\kin{N}^{-}$.
Thus $S$ is satisfied when $r^{\mcal{M}}$ is the rank function
of $\kin{N}^{-}$, contrary to our assumption.

\paragraph{\bf Case 2.}
Next we assume that $S$ has the form
\[
\exists X_{i_{1}}\cdots\exists X_{i_{m}}
\forall x_{j_{1}}\cdots\forall x_{j_{n}}P.
\]
As \kin{N} is representable over a field in \mcal{F}, there
is some function
\[\phi_{\mcal{S}}\colon \{X_{i}\}_{i\in I}\to
\mcal{P}(E^{\mcal{M}}),\]
such that for every possible function
\[\phi_{\mcal{E}}\colon \{x_{j}\}_{j\in J}\to E^{\mcal{M}},\]
$P(\phi_{\mcal{S}},\phi_{\mcal{E}})$ is satisfied
when $r^{\mcal{M}}$ is the rank function
of \kin{N}.

For every $k\in\{1,\ldots, N-1\}$, let
\[\mcal{H}_{k}=
\{(H_{k}\cup H_{N})\triangle Z\mid Z\subseteq E^{\mcal{M}},
\ |Z|\leq 2n\},\]
where $\triangle$ denotes symmetric difference.
If some subset of $E^{\mcal{M}}$ is contained in
$\mcal{H}_{k_{1}}$ and $\mcal{H}_{k_{2}}$, where
$k_{1}\ne k_{2}$, then
$(H_{k_{1}}\cup H_{N})\triangle Z_{1}=
(H_{k_{2}}\cup H_{k})\triangle Z_{2}$, for
some sets $Z_{1}$ and $Z_{2}$ satisfying $|Z_{1}|,|Z_{2}|\leq 2n$.
Thus
\[
\emptyset=
((H_{k_{1}}\cup H_{N})\triangle Z_{1})\triangle
((H_{k_{2}}\cup H_{N})\triangle Z_{2})
=(H_{k_{1}}\triangle H_{k_{2}})\triangle (Z_{1}\triangle Z_{2}).
\]
But $H_{k_{1}}\triangle H_{k_{2}}=H_{k_{1}}\cup H_{k_{2}}$,
and this set has cardinality $2N-4$.
Thus
\[
2^{2^{L}+1}+2=2N-4=|Z_{1}\triangle Z_{2}|\leq
|Z_{1}\cup Z_{2}|\leq
|Z_{1}|+|Z_{2}|\leq 4n,
\]
and this is impossible as $n\leq L$.
This shows that no subset of $E^{\mcal{M}}$ lies in two distinct
families in $\mcal{H}_{1},\ldots, \mcal{H}_{N-1}$.
The number of definable subsets relative to
$\phi_{\mcal{S}}$ is $2^{2^{m}}$, which is less than
$N-1$.
Let $s$ be an index in $\{1,\ldots, N-1\}$ such that no
member of $\mcal{H}_{s}$ is definable
relative to $\phi_{\mcal{S}}$.

Let $\phi_{\mcal{E}}\colon\{x_{j}\}_{j\in J}\to E^{\mcal{M}}$
be an arbitrary assignment of element variables.
Proposition~\ref{symdif} tells us that a definable
set relative to $(\phi_{\mcal{S}},\phi_{\mcal{E}})$
is obtained from a definable set relative
to $\phi_{\mcal{S}}$ by removing at most $n$ elements
and then adding at most $n$ elements.
That is, a definable set relative to
$(\phi_{\mcal{S}},\phi_{\mcal{E}})$ is the symmetric difference
of a definable set relative to $\phi_{\mcal{S}}$,
and a set of cardinality at most $2n$.
Thus $H_{s}\cup H_{N}$ is not definable in
$(\phi_{\mcal{S}},\phi_{\mcal{E}})$, for any choice of the
assignment $\phi_{\mcal{E}}$, or else some definable set
relative to $\phi_{\mcal{S}}$ would be in
$\mcal{H}_{s}$.

Let $M$ be the matroid obtained from \kin{N}
by relaxing the circuit-hyperplane $H_{s}\cup H_{N}$.
Then the rank functions of \kin{N} and $M$ differ
only in $H_{s}\cup H_{N}$.
Since $P(\phi_{\mcal{S}},\phi_{\mcal{E}})$ is satisfied
by \kin{N} for any choice of the function $\phi_{\mcal{E}}$,
it follows that $P(\phi_{\mcal{S}},\phi_{\mcal{E}})$ is also
satisfied by $M$ for
any assignment $\phi_{\mcal{E}}$.
Thus $S$ is satisfied when $r^{\mcal{M}}$ is the
rank function of $M$.
Clearly $M$ is isomorphic to $\kin{N}^{-}$, and it
is easy to show that $S$ is satisfied when $r^{\mcal{M}}$
is the rank function of $\kin{N}^{-}$,
contradicting our assumption.

\paragraph{\bf Case 3.}
Assume that $S$ has the form
\[
\forall X_{i_{1}}\cdots\forall X_{i_{m}}
\forall x_{j_{1}}\cdots\forall x_{j_{n}}P.
\]
Then there are functions, $\phi_{\mcal{S}}$ and $\phi_{\mcal{E}}$,
such that $P(\phi_{\mcal{S}},\phi_{\mcal{E}})$ is not satisfied
when $r^{\mcal{M}}$
is the rank function of $\kin{N}^{-}$.
Choose $s\in\{2,\ldots, N-1\}$ so that
$H_{s}\cup H_{N}$ is not definable relative to
$(\phi_{\mcal{S}},\phi_{\mcal{E}})$.
Then
$P(\phi_{\mcal{S}},\phi_{\mcal{E}})$ is also
not satisfied when $r^{\mcal{M}}$ is the
rank function of $\kin{N}_{s}^{=}$.
This means that $S$ is not satisfied by
$\kin{N}_{s}^{=}$, and this is a contradiction as
Lemma~\ref{doublerelax} implies that
$\kin{N}_{s}^{=}$ is representable over at least one
field in \mcal{F}.

\paragraph{\bf Case 4.}
In the final case, we assume that $S$ has the form
\[
\forall X_{i_{1}}\cdots\forall X_{i_{m}}
\exists x_{j_{1}}\cdots\exists x_{j_{n}}P.
\]
Let
\[\phi_{\mcal{S}}\colon \{X_{i}\}_{i\in I}\to
\mcal{P}(E^{\mcal{M}}),\]
be an assignment so that $P(\phi_{\mcal{S}},\phi_{\mcal{E}})$ is
not satisfied by $\kin{N}^{-}$ for every choice of assignment
\[\phi_{\mcal{E}}\colon \{x_{j}\}_{j\in J}\to E^{\mcal{M}}.\]
For $k\in \{2,\ldots, N-1\}$, we define
$\mcal{H}_{k}$ exactly as we did in Case~2.
Choose the index $s\in\{2,\ldots, N-1\}$ so that
no subset in $\mcal{H}_{s}$ is definable relative to
$\phi_{\mcal{S}}$.
Then $H_{s}\cup H_{N}$ is not definable relative to
$(\phi_{\mcal{S}},\phi_{\mcal{E}})$, for any choice of
assignment $\phi_{\mcal{E}}$.
This means that $P(\phi_{\mcal{S}},\phi_{\mcal{E}})$ is not
satisfied by $\kin{N}_{s}^{=}$,
for the assignment $\phi_{\mcal{S}}$ and any
choice of assignment $\phi_{\mcal{E}}$.
Thus $S$ is not satisfied when $r^{\mcal{M}}$ is the rank function of
$\kin{N}_{s}^{=}$, and as this matroid is in $M(\mcal{F})$,
we have reached a contradiction that completes the proof
of Theorem~\ref{main2}.
\end{proof}

\begin{remark}
We developed MSOL using the function $r$, which has an
intended interpretation as a rank function.
If we add a unary independence predicate, $I$, for
set terms, it is still not possible to finitely
axiomatize representability over
any infinite field, using sentences in \PH.
To see this, note that if there were such an axiomatization,
we could simply replace every occurrence of $I(X)$ with the
predicate $r(X)=|X|$.
Then we would have a contradiction to Theorem~\ref{main}.
The same comment applies when we add a predicate for bases
or spanning sets.
\end{remark}

\begin{remark}
The authors of \cite{MOOW11} conjecture that if \mcal{F}
is a collection of finite fields, then $M(\mcal{F})$ has a
finite number of excluded minors.
This would imply that membership in $M(\mcal{F})$ can always
be finitely axiomatized using sentences in \PH\ when
\mcal{F} contains no infinite field.
In other words, if the conjecture is true, then the
constraint in Theorem~\ref{main} that \mcal{F} contains an
infinite field is always necessary.
\end{remark}

We conclude with a conjecture that strengthens
Conjecture~\ref{conj}.

\begin{conjecture}
Theorem~{\rm \ref{main2}} holds even if
`\PH' is replaced by `MSOL'.
\end{conjecture}

\section{Acknowledgements}
We thank Rod Downey, Jim Geelen, Rob Goldblatt,
Noam Greenberg, Martin Grohe, and Johann Makowsky for
valuable discussions.



\end{document}